\documentclass[11pt]{amsart}

\author[C.~Sanna]{Carlo Sanna$^\dagger$}
\thanks{$\dagger\,$C.~Sanna is a member of GNSAGA of INdAM and of CrypTO, the group of Cryptography and Number~Theory of Politecnico di Torino}
\address{\parbox{\linewidth}{
Politecnico di Torino, Department of Mathematical Sciences\\
Corso Duca degli Abruzzi 24, 10129 Torino, Italy\\[-8pt]}}
\email{carlo.sanna.dev@gmail.com}



\title{On the divisibility of the rank of appearance of a Lucas~sequence}

\usepackage{amsmath}
\usepackage{amssymb}
\usepackage{amsthm}
\usepackage{geometry}
\usepackage{color}
\usepackage{hyperref}
\usepackage{enumerate}
\hypersetup{colorlinks=true}

\newtheorem{thm}{Theorem}[section]

\newtheorem{lem}[thm]{Lemma}
\theoremstyle{remark}

\newcommand{\Li}{\operatorname{Li}}
\newcommand{\id}{\operatorname{id}}

\newcommand{\Gal}{\operatorname{Gal}}
\newcommand{\disc}{\operatorname{disc}}

\begin{document}

\maketitle

\begin{abstract}
Let $U = (U_n)_{n \geq 0}$ be a Lucas sequence and, for every prime number $p$, let $\rho_U(p)$ be the rank of appearance of $p$ in $U$, that is, the smallest positive integer $k$ such that $p$ divides $U_k$, whenever it exists.
Furthermore, let $d$ be an odd positive integer.
Under some mild hypotheses, we prove an asymptotic formula for the number of primes $p \leq x$ such that $d$ divides $\rho_U(p)$, as $x \to +\infty$.
\end{abstract}

\section{Introduction}

Let $(U_n)_{n \geq 0}$ be a Lucas sequence, that is, a sequence of integers satisfying $U_0 = 0$, $U_1 = 1$, and $U_n = a_1 U_{n - 1} + a_2 U_{n - 2}$ for every integer $n \geq 2$, where $a_1, a_2$ are fixed nonzero integers.
The \emph{rank of appearance} of a prime number $p$, denoted by $\rho_U(p)$, is the smallest positive integer $k$ such that $p \mid U_k$.
It can be easily seen that $\rho_U(p)$ exists whenever $p \nmid a_2$.
Define
\begin{equation*}
\mathcal{R}_U(d; x) := \#\big\{p \leq x : p \nmid a_2, \, d \mid \rho_U(p) \big\} ,
\end{equation*}
for every positive integer $d$ and for every $x > 1$.

Let $(F_n)_{n \geq 0}$ be the Lucas sequence of Fibonacci numbers, corresponding to \mbox{$a_1 = a_2 = 1$.}
In~1985, Lagarias~\cite{MR789184} (see~\cite{MR1251907} for a correction and~\cite{MR1438482,MR1483692} for generalizations) showed that $\mathcal{R}_F(2; x) \sim \tfrac{2}{3} x$, as $x \to +\infty$.
More recently, Cubre and Rouse~\cite{MR3251719}, settling a conjecture of Bruckman and Anderson~\cite{MR1627443}, proved that $\mathcal{R}_F(d; x) \sim \mathrm{c}(d)\,d^{-1}\prod_{p \mid d}\left(1-p^{-2}\right)^{-1}$, as $x \to +\infty$, for every positive integer $d$, where $\mathrm{c}(d)$ is equal to $1$, $\tfrac{5}{4}$, or $\tfrac1{2}$, whenever $10 \nmid d$, $d \equiv 10 \pmod {20}$, or $20 \mid d$, respectively.

Let $\alpha, \beta$ be the roots of the characteristic polynomial $f_U(X) := X^2 - a_1 X - a_2$, and assume that $\gamma := \alpha / \beta$ is not a root of unity.
Let $\Delta := a_1^2 + 4a_2$ be the discriminant of $f_U(X)$, and let $\Delta_0$ be the squarefree part of $\Delta$.
Assume that $\Delta$ is not a square, so that $K := \mathbb{Q}\big(\!\sqrt{\Delta}\big)$ is a quadratic number field.
Let $h$ be the greatest positive integer such that $\gamma$ is a $h$th power in $K$.

Our result is the following:

\begin{thm}\label{thm:main}
Let $d$ be an odd positive integer with $3 \nmid d$ whenever $\Delta_0 = -3$.
Then, for every $x > \exp\!\big(Be^{8\omega(d)} d^8 \big)$, we have
\begin{equation*}
\mathcal{R}_U(d; x) = \delta_{U}(d) \Li(x) + O_U\!\left(\frac{(\omega(d) + 1)d}{\varphi(d)} \cdot \frac{x\,(\log \log x)^{\omega(d)}}{(\log x)^{9/8}}\right) ,
\end{equation*}
where $B > 0$ is an absolute constant and
\begin{equation*}
\delta_U(d) := \frac1{d}\left(\frac1{(d^\infty, h)} + \eta_U(d)\right) \prod_{p \,\mid\, d} \left(1 - \frac1{p^2}\right)^{-1} ,
\end{equation*}
with $\eta_U(d) := 0$ if $\Delta > 0$ or $\Delta_0 \not\equiv 1 \pmod 4$ or $\Delta_0 \nmid d^\infty$; and
\begin{equation*}
\eta_U(d) := \frac{(d^\infty, h)}{\big[(d^\infty, h), \Delta_0 / (d, \Delta_0)\big]^2}
\end{equation*}
otherwise.
\end{thm}

Cubre and Rouse's proof of the asymptotic formula for $\mathcal{R}_F(d; x)$ relies on the study of the algebraic group $G : x^2 - 5y^2 = 1$ and relates $\rho_F(p)$ with the order of $(3/2,1/2) \in G(\mathbb{F}_p)$.
Instead, our proof of Theorem~\ref{thm:main} is an adaptation of the methods that Moree~\cite{MR2274151} used to prove an asymptotic formula for the number of primes $p \leq x$ such that the multiplicative order of $g$ modulo $p$ is divisible by $d$, where $g \notin \{-1, 0, +1\}$ is a fixed rational number.

\section{Acknowledgements}

The author thanks Laura~Capuano (Politecnico di Torino) for several helpful discussions concerning Lemma~\ref{lem:Knd}.

\section{Notation}

We employ the Landau--Bachmann ``Big Oh'' notation $O$, as well as the associated Vinogradov symbol $\ll$. 
Any dependence of the implied constants is explicitly stated or indicated with subscripts. 
In particular, notations like $O_U$ and $\ll_U$ are shortcuts for $O_{a_1,a_2}$ and $\ll_{a_1,a_2}$, respectively.
For $x \geq 2$ we let $\Li(x) := \int_2^x \!\tfrac{\mathrm{d}t}{\log t}$ denote the logarithmic integral.
We reserve the letter $p$ for prime numbers.
Given an integer $d$, we let $d^\infty$ denote the supernatural number $\prod_{p \mid d} p^\infty$.
Given a field $F$ and a positive integer $n$, we write $F^n$ for the set of $n$th powers of elements of $F$.
Given a Galois extension $E/F$ of number fields and a prime ideal $P$ of $\mathcal{O}_E$ lying above an unramified prime ideal $\mathfrak{p}$ of $\mathcal{O}_F$, we write $\Big[\frac{E / F}{P}\Big]$ for the Frobenius automorphism corresponding to $P / \mathfrak{p}$, that is, the unique element $\sigma$ of the Galois group $\Gal(E / F)$ that satisfies $\sigma(a) \equiv a^{N(\mathfrak{p})} \pmod P$ for every $a \in \mathcal{O}_E$, where $N(\mathfrak{p})$ denotes the norm of $\mathfrak{p}$.
Moreover, we let $\Big[\frac{E / F}{\mathfrak{p}}\Big]$ be the set of all $\Big[\frac{E / F}{P}\Big]$ with $P$ prime ideal of $\mathcal{O}_E$ lying over $\mathfrak{p}$.
We write $\Delta_{E/F}$ for the relative discriminant of $E/F$, and $\Delta_E := \Delta_{E/\mathbb{Q}}$ for the absolute discriminant of $E$.
For every integer $d$ and for every prime number $p$ we let $\big(\tfrac{d}{p}\big)$ be the Legendre symbol.
For every positive integer $n$, we let $\zeta_n := \mathrm{e}^{2\pi\mathbf{i} / n}$ be a primitive $n$th root of unity.
We write $\omega(n)$, $\varphi(n)$, $\mu(n)$, and $\tau(n)$, for the number of prime factors, the totient function, the M\"obius function, and the number of divisors of a positive integer $n$, respectively.

\section{General preliminaries}

\begin{lem}\label{lem:zetanmodP}
Let $n$ be a positive integer, let $p$ be a prime number not dividing $n$, and let $P$ be a prime ideal of $\mathcal{O}_{\mathbb{Q}(\zeta_n)}$ lying over $p$.
Then $\zeta_n$ has multiplicative order modulo $P$ equal to $n$.
\end{lem}
\begin{proof}
Let $k$ be the multiplicative order of $\zeta_n$ modulo $P$, that is, $k$ is the least positive integer such that $\zeta_n^k \equiv 1 \pmod P$.
On the one hand, we have that $p \mid N(P) \mid N(\zeta_n^k - 1)$.
On the other hand, since $\zeta_n^n \equiv 1 \pmod P$, we have that $k \mid n$, and consequently $\zeta_n^k$ is a $m$th primitive root of unity, where $m := n / k$.
If $k < n$ then $m > 1$ and $N(\zeta_n^k - 1)$ is either $1$ or a prime factor of $m$, but both cases are impossible since $p \nmid n$.
Hence, $k = n$.
\end{proof}

\begin{lem}\label{lem:binomial}
Let $F$ be a field, let $a \in F$, and let $n$ be a positive integer.
Then $X^n - a$ is irreducible over $F$ if and only if $a \notin F^p$ for each prime $p$ dividing $n$ and $a \notin -4F^4$ \mbox{whenever $4 \mid n$.}
\end{lem}
\begin{proof}
See~\cite[Chapter~8, Theorem~1.6]{MR982265}.
\end{proof}

\begin{lem}\label{lem:abelian}
Let $F$ be a field, let $n$ be a positive integer not divisible by the characteristic of $F$, and let $m$ be the number of $n$th roots of unity contained in $F$.
Then, for every $a \in F$, the extension $F\big(\zeta_n, a^{1/n}\big) / F$ is abelian if and only if $a^m \in F^n$.
\end{lem}
\begin{proof}
See~\cite[Chapter~8, Theorem~3.2]{MR982265}.
\end{proof}

\begin{lem}\label{lem:sqrt}
Let $n$ be an odd positive integer and let $d$ be a squarefree integer.
Then $\sqrt{d} \in \mathbb{Q}(\zeta_n)$ if and only if $d \mid n$ and $d \equiv 1 \pmod 4$.
\end{lem}
\begin{proof}
See~\cite[Lemma~3]{MR284417}.
\end{proof}

We need the following form of the Chebotarev Density Theorem.

\begin{thm}\label{thm:chebotarev}
Let $E/F$ be a Galois extension of numbers fields with Galois group $G$, and let $C$ be the union of $k$ conjugacy classes of $G$.
Then
\begin{align*}
&\#\!\left\{\mathfrak{p} \text{ prime ideal of } \mathcal{O}_F \text{ non-ramifying in } E :  N_{F/\mathbb{Q}}(\mathfrak{p}) \leq x,\, \Big[\tfrac{E/F}{\mathfrak{p}}\Big] \subseteq C\right\} \\
&\phantom{mmmmm}=\frac{\#C}{\#G} \cdot \Li(x) + O\!\left(k\;\!x \exp\!\Big({-c_1}\big(\log x / n_E\big)^{1/2}\Big)\right)
\end{align*}
for every
\begin{equation*}
x \geq \exp\!\left(c_2 \max\!\Big(n_E(\log |\Delta_E|)^2, |\Delta_E|^{2/n_E} / n_E\Big)\right) ,
\end{equation*}
where $n_E := [E : \mathbb{Q}]$ and $c_1, c_2 > 0$ are absolute constants.
\end{thm}
\begin{proof}
The result follows from the effective form of the Chebotarev Density Theorem given by Lagarias and Odlyzko~\cite[Theorem~1.3]{MR0447191} and from the bounds for the exceptional zero of the Dedekind zeta function $\zeta_E$ given by Stark~\cite[Lemma~8 and~11]{MR342472}.
\end{proof}

\section{Preliminaries to the proof of Theorem~\ref{thm:main}}

Recalling that $h$ is the greatest positive integer such that $\gamma$ is an $h$th power in $K$, write $\gamma = \gamma_0^h$ for some $\gamma_0 \in K$.
Also, let $\sigma_K \in \Gal(K / \mathbb{Q})$ be the nontrivial automorphism, which satisfies $\sigma_K\big(\!\sqrt{\Delta}\;\!\big) = -\sqrt{\Delta}$.
Note that, since $\gamma = \alpha/\beta$ and $\sigma_K$ swaps $\alpha$ and $\beta$, we have that $\sigma_k(\gamma) = \gamma^{-1}$.
For all positive integers $d, n$ such that $d \mid n$, let $K_{n, d} := K\big(\zeta_n, \gamma^{1/d}\big)$.

\begin{lem}\label{lem:order}
Let $p$ be a prime number not dividing $a_2 \Delta$ and let $\pi$ be a prime ideal of $\mathcal{O}_K$ lying over $p$.
Then $\rho_U(p)$ is equal to the multiplicative order of $\gamma$ modulo $\pi$.
Moreover, $\rho_U(p)$ divides $p - \big(\tfrac{\Delta}{p}\big)$.
\end{lem}
\begin{proof}
First, note that $p \nmid a_2$ ensures that $\beta$ is invertible modulo $\pi$, and consequently it makes sense to consider the multiplicative order of $\gamma = \alpha / \beta$ modulo $\pi$.
Also, $p \nmid \Delta$ implies that $p$ does not ramifies in $K$ and that $\alpha \not\equiv \beta \pmod \pi$.

We shall prove that $p \mid U_n$ if and only if $\gamma^n \equiv 1 \pmod \pi$, for every positive integer $n$.
Then the claim on $\rho_U(p)$ follows easily.
It is well known that the Binet's formula
\begin{equation}\label{equ:binet}
U_n = \frac{\alpha^n - \beta^n}{\alpha - \beta}
\end{equation}
holds for every positive integer $n$.
On the one hand, if $p \mid U_n$ then, since $p \mathcal{O}_K \subseteq \pi$ and~\eqref{equ:binet}, we have $\alpha^n \equiv \beta^n \pmod \pi$, and consequently $\gamma^n \equiv 1 \pmod \pi$.
On the other hand, if $\gamma^n \equiv 1 \pmod \pi$ then by~\eqref{equ:binet} we get $U_n \equiv 0 \pmod \pi$.
If $p$ is inert in $K$, then $p\mathcal{O}_K = \pi$ and so $p \mid U_n$.
If $p$ splits in $K$, then $p\mathcal{O}_K = \pi \cap \sigma_K(\pi)$.
Thus $U_n \equiv 0 \pmod \pi$ and $U_n \equiv \sigma_K(U_k) \equiv 0 \pmod{\sigma_K(\pi)}$ imply that $p \mid U_n$.

Let $\sigma := \Big[\tfrac{K / \mathbb{Q}}{\pi}\Big]$.
On the one hand, if $\big(\tfrac{\Delta}{p}) = -1$ then $\sigma = \sigma_K$ and $\gamma^{p + 1} \equiv \sigma_K(\gamma) \gamma \equiv \gamma^{-1} \gamma \equiv 1 \pmod \pi$, so that $\rho_U(p) \mid p + 1$.
On the other hand, if $\big(\tfrac{\Delta}{p}) = +1$ then $\sigma = \id$ and $\gamma^{p - 1} \equiv \gamma \gamma^{-1} \equiv 1 \pmod \pi$, so that $\rho_U(p) \mid p - 1$.
\end{proof}

For each prime number $p$ not dividing $a_2 \Delta$, let us define the \emph{index of appearance} of $p$ as
\begin{equation*}
\iota_U(p) := \big(p - \big(\tfrac{\Delta}{p}\big)\big) / \rho_U(p) .
\end{equation*}
Note that, in light of Lemma~\ref{lem:order}, $\iota_U(p)$ is an integer.

\begin{lem}\label{lem:iota}
Let $d, n$ be positive integers such that $d \mid n$, and let $p$ be a prime number not dividing $a_2 \Delta$.
Moreover, let $P$ be a prime ideal of $\mathcal{O}_{K_{n,d}}$ lying over $p$ and let $\sigma := \Big[\tfrac{K_{n,d} / \mathbb{Q}}{P}\Big]$.
Then
\begin{equation}\label{equ:diotaU}
p \equiv \big(\tfrac{\Delta}{p}\big) \!\!\!\!\pmod n \quad\text{ and }\quad d \mid \iota_U(p)
\end{equation}
if and only if $\sigma = \id$ or
\begin{equation}\label{equ:sigma}
\sigma(\zeta_n) = \zeta_n^{-1} \quad\text{ and }\quad \sigma\big(\gamma^{1/d}\big) = \gamma^{-1/d} . 
\end{equation}
\end{lem}
\begin{proof}
First, suppose that $\big(\tfrac{\Delta}{p}\big) = -1$.
Let us assume~\eqref{equ:diotaU}.
On the one hand, since $p \equiv -1 \pmod n$, we have
\begin{equation}\label{equ:sigmazetan1}
\sigma(\zeta_n) \equiv \zeta_n^p \equiv \zeta_n^{-1} \pmod P .
\end{equation}
Since $\sigma(\zeta_n) = \zeta_n^k$ for some integer $k$, and since $p$ does not divide $n$, Lemma~\ref{lem:zetanmodP} and~\eqref{equ:sigmazetan1} yield that $\sigma(\zeta_n) = \zeta_n^{-1}$.

On the other hand, $d \mid \iota_U(p)$ implies that $\rho_U(p) \mid (p + 1) / d$.
Hence, letting $\pi := P \cap \mathcal{O}_K$, Lemma~\ref{lem:order} yields $\gamma^{(p + 1) / d} \equiv 1 \pmod \pi$.
Consequently,
\begin{equation}\label{equ:sigmagammad}
\sigma\big(\gamma^{1/d}\big) \equiv \big(\gamma^{1/d}\big)^p \equiv \gamma^{(p + 1)/d} \cdot \gamma^{-1/d} \equiv \gamma^{-1/d} \pmod P .
\end{equation}
Note that, since $\big(\tfrac{\Delta}{p}\big) = -1$, we have 
\begin{equation*}
\sigma(\gamma) = \sigma|_K(\gamma) = \Big[\tfrac{K / \mathbb{Q}}{\pi}\Big](\gamma) = \sigma_K(\gamma) = \gamma^{-1},
\end{equation*}
so that $\sigma\big(\gamma^{1/d}\big) = \zeta_d^k \gamma^{-1/d}$ for some integer $k$.
Thus Lemma~\ref{lem:zetanmodP} and~\eqref{equ:sigmagammad} yield that $\sigma\big(\gamma^{1/d}\big) = \gamma^{-1/d}$.
We have proved~\eqref{equ:sigma}.

Now let us assume~\eqref{equ:sigma}.
On the one hand, we have
\begin{equation*}
\zeta_n^{-1} = \sigma(\zeta_n) = \sigma|_{\mathbb{Q}(\zeta_n)}(\zeta_n) = \left[\frac{\mathbb{Q}(\zeta_n) / \mathbb{Q}}{P \cap \mathcal{O}_{\mathbb{Q}(\zeta_n)}}\right]\!(\zeta_n) = \zeta_n^p ,
\end{equation*}
so that $p \equiv -1 \pmod n$.
On the other hand,
\begin{equation*}
\gamma^{(p + 1) / d} \equiv \big(\gamma^{1/d}\big)^p \cdot \gamma^{1/d} \equiv \sigma\big(\gamma^{1/d}\big) \cdot \gamma^{1/d} \equiv \gamma^{-1/d} \cdot \gamma^{1/d} \equiv 1 \pmod P ,
\end{equation*}
so that $\gamma^{(p + 1) / d} \equiv 1 \pmod \pi$, which, by Lemma~\ref{lem:order}, implies $d \mid \iota_U(p)$.
We have proved~\eqref{equ:diotaU}.

If $\big(\tfrac{\Delta}{p}\big) = +1$ then the proof proceeds similarly to the case $\big(\tfrac{\Delta}{p}\big) = -1$, and yields that~\eqref{equ:diotaU} is equivalent to $\sigma(\zeta_n) = \zeta_n$ and $\sigma\big(\gamma^{1/d}\big) = \gamma^{1/d}$, that is, $\sigma = \id$.
\end{proof}

\begin{lem}\label{lem:rootsof1inK}
The roots of unity contained in $K$ are: the sixth roots of unity, if $\Delta_0 = -3$; the forth roots of unity, if $\Delta_0 = -1$; or the second roots of unity, if $\Delta_0 \neq -1, -3$.
\end{lem}
\begin{proof}
If $\zeta_n \in K$ for some positive integer $n$, then $\mathbb{Q}(\zeta_n) \subseteq K$, so that $\varphi(n) \leq 2$, and $n \in \{1, 2, 3, 4, 6\}$.
Then the claim follows easily since $\zeta_3 = (-1 + \sqrt{-3}) / 2$, $\zeta_4 = \sqrt{-1}$, and $\zeta_6 = (1 + \sqrt{-3}) / 2$.
\end{proof}

\begin{lem}\label{lem:powers}
Let $n$ be an odd positive integer with $3 \nmid n$ whenever $\Delta_0 = -3$, and let $d$ be a positive integer dividing $n$.
Then $a \in K \cap K(\zeta_n)^d$ if and only if $a \in K^d$.
\end{lem}
\begin{proof}
The ``if'' part if obvious.
Let us prove the ``only if'' part.
Note that, by the hypothesis on $n$ and by Lemma~\ref{lem:rootsof1inK}, the only $n$th root of unity in $K$ is $1$.
Suppose that $a \in K \cap K(\zeta_n)^d$.
Hence, there exists $b \in K(\zeta_n)$ such that $a = b^d$.
Putting $a_1 := a^{n / d}$, we get that $a_1 = b^n$.
Therefore, $K\big(\zeta_n, a_1^{1/n}\big) = K(\zeta_n, b) = K(\zeta_n)$ is an abelian extension of $K$.
Consequently, by Lemma~\ref{lem:abelian}, we have $a_1 \in K^n$, that is, $a_1 = b_1^n$ for some $b_1 \in K$.
Thus $a^n = a_1^d = b_1^{dn}$, so that $a = \zeta b_1^d$, where $\zeta$ is a $n$th root of unity in $K$.
We already noticed that $\zeta = 1$, hence $a \in K^d$.
\end{proof}

\begin{lem}\label{lem:Knd}
Let $n$ be an odd positive integer with $3 \nmid n$ whenever $\Delta_0 = -3$, and let $d$ be a positive integer dividing $n$.
Then
\begin{equation}\label{equ:degreeKnd}
[K_{n,d} : \mathbb{Q}] = \frac{\varphi(n)d}{(d, h)} \cdot
\begin{cases}
1 & \text{ if } \sqrt{\Delta} \in \mathbb{Q}(\zeta_n) , \\
2 & \text{ if } \sqrt{\Delta} \notin \mathbb{Q}(\zeta_n) , \\
\end{cases}
\end{equation}
while
\begin{equation}\label{equ:Knddisc}
|\Delta_{K_{n,d}}|^{1 / [K_{n,d} : \mathbb{Q}]} \ll_U n^3 \quad\text{ and }\quad \log|\Delta_{K_{n,d}}| \ll_U n^2 \log(n + 1) .
\end{equation}
Moreover, there exists $\sigma \in \Gal(K_{n,d} / \mathbb{Q})$ satisfying~\eqref{equ:sigma} if and only if $\sqrt{\Delta} \notin \mathbb{Q}(\zeta_n)$ or $\Delta < 0$.
In particular, if $\sigma$ exists then it belongs to the center of $\Gal(K_{n,d} / \mathbb{Q})$.
\end{lem}

\begin{proof}
Let $d_0 := d / (d, h)$, $h_0 := h / (d, h)$, and $f(X) = X^{d_0} - \gamma_0^{h_0}$.
Suppose that $\gamma_0^{h_0} \in K(\zeta_n)^p$ for some prime number $p$ dividing $d_0$.
Then, by Lemma~\ref{lem:powers}, we have $\gamma_0^{h_0} \in K^p$.
In turn, by the maximality of $h$, it follows that $p \mid h_0$, which is impossible, since $(d_0, h_0) = 1$.
Hence, $\gamma_0^{h_0} \notin K(\zeta_n)^p$ for every prime number $p$ dividing $d_0$.
Consequently, by Lemma~\ref{lem:binomial}, $f$ is irreducible over $K(\zeta_n)$.
Thus $K_{n,d} \cong K(\zeta_n)[X] / (f(X))$, so that $[K_{n,d} : K(\zeta_n)] = d_0$ and $\big(\gamma^{1/d}\big)^{d_0} = \gamma_0^{h_0}$.
It is easy to check that $[K(\zeta_n) : \mathbb{Q}] = \varphi(n)$ if $\sqrt{\Delta} \in \mathbb{Q}(\zeta_n)$, and $[K(\zeta_n) : \mathbb{Q}] = 2\varphi(n)$ otherwise.
Hence,~\eqref{equ:degreeKnd} follows.

Let $s$ be a positive integer such that $s \gamma_0 \in \mathcal{O}_K$, and put $g(X) := s^{d_0} f(X/s) = X^{d_0} - s^{d_0} \gamma_0^{h_0}$.
Since $f$ is the minimal polynomial of $\gamma^{1/d}$ over $K(\zeta_n)$, we get that $g$ is the minimal polynomial of $s\gamma^{1/d}$ over $K(\zeta_n)$.
In particular, since $g \in \mathcal{O}_K[X]$, we have that $s\gamma^{1/d} \in \mathcal{O}_{K_{n,d}}$, .
Hence, from $K_{n,d} = K(\zeta_n)\big(s\gamma^{1/d}\big)$ it follows that
\begin{align*}
\Delta_{K_{n,d} / K(\zeta_n)} &\supseteq \disc(g)\, \mathcal{O}_{K(\zeta_n)} = \prod_{1 \,\leq\, i \,<\, j \,\leq\, d_0} \big(s \gamma^{1/d} \zeta_{d_0}^i - s \gamma^{1/d} \zeta_{d_0}^j\big)^2 \mathcal{O}_{K(\zeta_n)} \\
&= \big(s\gamma^{1/d}\big)^{d_0(d_0 - 1)} d_0^{d_0} \mathcal{O}_{K(\zeta_n)} = \gamma_0^{h_0(d_0 - 1)} \big(s^{d_0 - 1}d_0\big)^{d_0} \mathcal{O}_{K(\zeta_n)} ,
\end{align*}
and
\begin{equation*}
N_{K(\zeta_n) / \mathbb{Q}}\big(\Delta_{K_{n,d} / K(\zeta_n)}\big) = N_{K / \mathbb{Q}}\big(\gamma_0^{h_0}\big)^{(d_0-1)[K(\zeta_n) : K]} \big(s^{d_0 - 1}d_0\big)^{d_0[K(\zeta_n) : \mathbb{Q}]} \mid \big(N_{K/\mathbb{Q}}(\gamma)sn\big)^\infty .
\end{equation*}
Also, a quick computation shows that $\Delta_{K(\zeta_n)} \mid (4\Delta n)^{\infty}$.
Therefore, since
\begin{equation*}
\Delta_{K_{n,d}} = \Delta_{K(\zeta_n)}^{[K_{n,d} : K(\zeta_n)]} N_{K(\zeta_n) / \mathbb{Q}}\big(\Delta_{K_{n,d} / K(\zeta_n)}\big) ,
\end{equation*}
we get that every prime factor of $\Delta_{K_{n,d}}$ divides $A n$, where $A := 4\Delta N_{K/\mathbb{Q}}(\gamma)s$.
By Hensel's estimate~(see, e.g.,~\cite[comments after Theorem~7.3]{MR1482805}), we have that
\begin{equation*}
|\Delta_L|^{1/n_L} \leq n_L \prod_{p \,\mid\, \Delta_L} p  ,
\end{equation*}
for every Galois extension $L / \mathbb{Q}$ of degree $n_L$.
Consequently,
\begin{equation*}
|\Delta_{K_{n,d}}|^{1/[K_{n,d} : \mathbb{Q}]} \leq [K_{n,d} : \mathbb{Q}] A n \ll_U \varphi(n)d n \leq n^3 ,
\end{equation*}
and
\begin{equation*}
\log|\Delta_{K_{n,d}}| \leq [K_{n,d} : \mathbb{Q}] \big(\log(n^3) + O_U(1)\big) \ll_U \varphi(n)d \log(n + 1) \ll n^2 \log(n + 1) ,
\end{equation*}
so that~\eqref{equ:Knddisc} is proved.

Suppose that there exists $\sigma \in \Gal(K_{n,d}/\mathbb{Q})$ satisfying~\eqref{equ:sigma}.
We shall prove that $\sqrt{\Delta} \notin \mathbb{Q}(\zeta_n)$ or $\Delta < 0$.
Assume that $\sqrt{\Delta} \in \mathbb{Q}(\zeta_n)$.
On the one hand, $\sigma(\gamma) = \sigma\big(\gamma^{1/d}\big)^d = \gamma^{-1}$, and consequently $\sigma\big(\!\sqrt{\Delta}\big) = -\sqrt{\Delta}$.
On the other hand, since $\sqrt{\Delta} \in \mathbb{Q}(\zeta_n)$ and $\sigma(\zeta_n) = \zeta_n^{-1}$, we have that $\sigma\big(\!\sqrt{\Delta}\big) = \overline{\sqrt{\Delta}}$.
Therefore, $\overline{\sqrt{\Delta}} = -\sqrt{\Delta}$ and so $\Delta < 0$.
Now let us check that $\sigma$ belongs to the center of $\Gal(K_{n,d} / \mathbb{Q})$.
Note that $N_{K/\mathbb{Q}}(\gamma) = \gamma \,\sigma_K(\gamma) = \gamma \gamma^{-1} = 1$.
Also, $N_{K/\mathbb{Q}}(\gamma_0^{h_0}) = N_{K/\mathbb{Q}}(\gamma_0^{h}) = N_{K/\mathbb{Q}}(\gamma) = 1$, since $d$ is odd and so $h_0 \equiv h \pmod 2$.
Therefore, for every $\tau \in \Gal(K_{n,q} / \mathbb{Q})$, we have $\tau(\gamma_0^{h_0}) = \gamma_0^{h_0}$, if $\tau|_K = \id$, or $\tau(\gamma_0^{h_0}) = N_{K/\mathbb{Q}}(\gamma_0^{h_0})\gamma_0^{-h_0} = \gamma_0^{-h_0}$ if $\tau|_K = \sigma_K$.
Consequently, recalling that $\big(\gamma^{1/d}\big)^{d_0} = \gamma_0^{h_0}$, we have that $\tau(\zeta_n) = \zeta_n^s$ and $\tau\big(\gamma^{1/d}\big) = \zeta_{d_0}^t \gamma^{\pm 1/d}$ for some integers $s, t$.
At this point, it can be easily checked that $(\sigma\tau)(\zeta_n) = (\tau\sigma)(\zeta_n)$ and $(\sigma\tau)\big(\gamma^{1/d}\big) = (\tau\sigma)\big(\gamma^{1/d}\big)$.
Hence, $\sigma$ belongs to the center of $\Gal(K_{n,d}/\mathbb{Q})$.

Suppose that $\sqrt{\Delta} \notin \mathbb{Q}(\zeta_n)$ or $\Delta < 0$.
We shall prove the existence of $\sigma \in \Gal(K_{n,d}/\mathbb{Q})$ satisfying~\eqref{equ:sigma}.
It suffices to show that there exists $\sigma_1 \in \Gal(K(\zeta_n) / K)$ such that $\sigma_1(\zeta_n) = \zeta_n^{-1}$ and $\sigma_1|_K = \sigma_K$.
Indeed, recalling that $K_{n,d} \cong K(\zeta_n)[X] / (f(X))$, we can extend $\sigma_1$ to an automorphism $\sigma \in \Gal(K_{n,d} / \mathbb{Q})$ that sends the root $\gamma^{1/d}$ of $f$ to the root $\gamma^{-1/d}$ of
\begin{equation*}
(\sigma_1 f)(X) = X^{d_0} - \sigma_1(\gamma_0^{h_0}) = X^{d_0} - N_{K / \mathbb{Q}}(\gamma_0^{h_0}) \gamma_0^{-h_0} = X^{d_0} - \gamma_0^{-h_0} ,
\end{equation*}
and so $\sigma$ satisfies~\eqref{equ:sigma}.
Pick $\sigma_0 \in \Gal(\mathbb{Q}(\zeta_n)/\mathbb{Q})$ such that $\sigma_0(\zeta_n) = \zeta_n^{-1}$.
If $\sqrt{\Delta} \in \mathbb{Q}(\zeta_n)$ then $K(\zeta_n) = \mathbb{Q}(\zeta_n)$, $\Delta < 0$, and $\sigma_0\big(\!\sqrt{\Delta}\big) = \overline{\sqrt{\Delta}} = -\sqrt{\Delta}$, so we let $\sigma_1 := \sigma_0$.
If $\sqrt{\Delta} \notin \mathbb{Q}(\zeta_n)$ then $X^2 - \Delta$ is the minimal polynomial of $\sqrt{\Delta}$ over $\mathbb{Q}(\zeta_n)$ and we can extend $\sigma_0$ to $\sigma_1 \in \Gal(K(\zeta_n) / \mathbb{Q})$ such that $\sigma_1\big(\!\sqrt{\Delta}\big) = -\sqrt{\Delta}$.
\end{proof}

\section{Proof of Theorem~\ref{thm:main}}

The proof proceeds similarly to~\cite[Section~2]{MR2274151}.
For all positive integers $d, n$ with $d \mid n$, and for all $x > 1$, let us define
\begin{equation*}
\pi_{U, n, d}(x) := \#\big\{p \leq x : p \nmid a_2\Delta,\, p \equiv \big(\tfrac{\Delta}{p}\big) \!\!\! \pmod n, \, d \mid \iota_U(p)\big\} .
\end{equation*}
In what follows, we will tacitly ignore the finitely many prime numbers dividing $a_2 \Delta$.

\begin{lem}\label{lem:RUdx}
For every positive integer $d$ and for every $x > 1$, we have
\begin{equation}\label{equ:Rdoublesum}
\mathcal{R}_U(d; x) = \!\!\!\sum_{\;\;\;v \,\mid\, d^\infty} \sum_{a \,\mid\, d} \mu(a) \pi_{U, dv, av}(x) .
\end{equation}
\end{lem}
\begin{proof}
Every prime number $p$ counted by the inner sum of~\eqref{equ:Rdoublesum} satisfies $p \leq x$, $p \equiv \big(\tfrac{\Delta}{p}\big) \pmod{dv}$, and $\iota_U(p) = v w$ for some integer $w$.
Writing $w = w_1 w_2$, with $w_1 := (w, d)$, we get that the contribution of $p$ to the inner sum or~\eqref{equ:Rdoublesum} is equal to $\sum_{a \mid w_1} \mu(a)$.
Hence,
\begin{equation}\label{equ:Rsinglesum}
\sum_{a \,\mid\, d} \mu(a) \pi_{U, dv, av}(x) = \#\big\{p \leq x : p \equiv \big(\tfrac{\Delta}{p}\big) \!\!\!\pmod{dv}, \, v \mid \iota_U(p), \, \big(\iota_U(p) / v, d\big) = 1\big\} .
\end{equation}
Now it suffices to show that
\begin{equation}\label{equ:Rdoublesum1}
\mathcal{R}_U(d; x) = \!\!\!\sum_{\;\;\;v \,\mid\, d^\infty} \#\big\{p \leq x : p \equiv \big(\tfrac{\Delta}{p}\big) \!\!\!\pmod{dv}, \, v \mid \iota_U(p), \, \big(\iota_U(p) / v, d\big) = 1\big\} .
\end{equation}
On the one hand, let $p$ be a prime number counted on the right-hand side of~\eqref{equ:Rdoublesum1}.
Note that this is counted only one, namely for $v = (\iota_U(p), d^\infty)$.
Then, from $\rho_U(p) \iota_U(p) = p - \big(\tfrac{\Delta}{p}\big)$, it follows that $d \mid \rho_U(p)$.
Hence, $p$ is counted on the left-hand side of~\eqref{equ:Rdoublesum1}.

On the other hand, let $p$ be a prime number counted by $\mathcal{R}_U(d; x)$.
Then $d \mid \rho_U(p)$ and, by Lemma~\ref{lem:order}, $p \equiv \big(\tfrac{\Delta}{p}\big) \pmod d$.
Consequently, there is an integer $v$ such that $v \mid d^\infty$, $p \equiv \big(\tfrac{\Delta}{p}\big) \pmod {dv}$, and $\big(\iota_U(p) / v, d\big) = 1$.
Hence, $p$ is counted on the right-hand side of~\eqref{equ:Rdoublesum1}.
\end{proof}

\begin{lem}\label{lem:piUnd}
Let $n$ be an odd positive integer with $3 \nmid n$ whenever $\Delta_0 = -3$, and let $d$ be a positive integer dividing $n$.
There exist absolute constants $A, B > 0$ such that
\begin{equation*}
\pi_{U, n, d}(x) = \delta_{U,n,d} \Li(x) + O_U\!\left(x \exp\big({-A}(\log x)^{1/2} / n \big)\right)
\end{equation*}
for $x \geq \exp(B n^8)$, where
\begin{equation}\label{equ:defdeltaUnd}
\delta_{U,n,d} := \frac{(d, h)}{\varphi(n)d} \cdot \begin{cases} 1 &\text{ if $\Delta > 0$ or $\Delta_0 \not\equiv 1 \!\!\!\pmod 4$ or $\Delta_0 \nmid n$},\\ 2 &\text{ otherwise}. \end{cases}
\end{equation}
\end{lem}
\begin{proof}
Put $E := K_{n,d}$, $F := \mathbb{Q}$, $G := \Gal(E/F)$, and $C = \{\id, \sigma\}$ if there exists $\sigma \in \Gal(K_{n,d} / \mathbb{Q})$ satisfying~\eqref{equ:sigma}, or $C = \{\id\}$ otherwise.
By Lemma~\ref{lem:Knd}, $\sigma$ belongs to the center of $G$, so that $C$ is the union of conjugacy classes of $G$.
By Lemma~\ref{lem:iota}, we have that $\pi_{U,n,d}(x)$ is the number of primes $p$ not exceeding $x$ and such that $\Big[\frac{E/F}{p}\Big] \subseteq C$.
Thus, taking into account the bounds for the degree and the discriminant of $E / F$ given in Lemma~\ref{lem:Knd}, and considering Lemma~\ref{lem:sqrt}, the asymptotic formula follows by applying Theorem~\ref{thm:chebotarev}.
\end{proof}

\begin{lem}\label{lem:bounds}
Let $d$ be an odd positive integer with $3 \nmid d$ whenever $\Delta_0 = -3$.
If $x > 1$ and $e^{\omega(d)} \leq y \leq \log x / \varphi(d)$, then
\begin{equation}\label{equ:firstbound}
\sum_{\substack{\;\;\;v \,\mid\, d^\infty \\[1pt] v \,>\, y}} \sum_{a \,\mid\, d} \mu(a) \pi_{U, dv, av}(x) \ll \frac{x}{\log x} \cdot \frac{\omega(d) + 1}{\varphi(d)} \cdot \frac{(\log y)^{\omega(d)}}{y}
\end{equation}
and
\begin{equation*}
\sum_{\substack{\;\;\;v \,\mid\, d^\infty \\[1pt] v \,>\, y}} \sum_{a \,\mid\, d} \mu(a) \delta_{U, dv, av} \ll_U \frac{\omega(d) + 1}{\varphi(d)} \cdot \frac{(\log y)^{\omega(d)}}{y} .
\end{equation*}
\end{lem}
\begin{proof}
Let $\pi(m, r; x) := \#\{p \leq x : p \equiv r \pmod m\}$.
From~\eqref{equ:Rsinglesum} it follows that
\begin{equation}\label{equ:sumadmuboundpi}
\left|\sum_{a \,\mid\, d} \mu(a) \pi_{U, dv, av}(x)\right| \leq \pi_{U, dv, v}(x) \leq \pi(x; dv, \pm 1) .
\end{equation}
Moreover, letting $x \to +\infty$, Lemma~\ref{lem:piUnd} and the first inequality of~\eqref{equ:sumadmuboundpi} yield
\begin{equation}\label{equ:sumadmubounddelta}
\left|\sum_{a \,\mid\, d} \mu(a) \delta_{U, dv, av}\right| \leq \delta_{U,dv,v} .
\end{equation}
Now we have $M_d(x) := \#\{v \leq x : v \mid d^\infty\} \ll (\log x)^{\omega(d)}$, for every $x \geq 2$.
Hence, by partial summation and since $y \geq e^{\omega(d)}$, we obtain that
\begin{equation}\label{equ:1overv}
\sum_{\substack{\;\;\;v \,\mid\, d^\infty \\[1pt] v \,>\, y}} \frac1{v} = \left.\frac{M_d(t)}{t}\right|_{t\,=\,y}^{+\infty} + \int_y^{+\infty} \frac{M_d(t)}{t^2}\,\mathrm{d}t \ll \int_y^{+\infty} \frac{(\log t)^{\omega(d)}}{t^2}\,\mathrm{d}t \leq \frac{(\omega(d) + 1)(\log y)^{\omega(d)}}{y} .
\end{equation}
On the one hand, using the Brun--Titchmarsh inequality~\cite[Theorem~12.7]{MR2919246}
\begin{equation*}
\pi(m, r; x) \ll \frac{x}{\varphi(m)\log(x / m)},
\end{equation*}
holding for $x > m$, and~\eqref{equ:1overv} we get that
\begin{equation}\label{equ:hand1}
\sum_{\substack{\;\;\;v \,\mid\, d^\infty \\[1pt] \;\;\;v \,>\, y,\; dv \,\leq\, x^{2/3}}} \pi(dv, \pm 1; x) \ll \frac{x}{\varphi(d)\log x} \!\!\!\sum_{\substack{\;\;\;v \,\mid\, d^\infty \\[1pt] v \,>\, y}} \frac1{v} \ll \frac{x}{\log x} \cdot \frac{\omega(d) + 1}{\varphi(d)} \cdot \frac{(\log y)^{\omega(d)}}{y} .
\end{equation}
On the other hand, using the trivial bound $\pi(m, \pm 1; x) \ll x / m$, holding for $x \geq 1$, and~\eqref{equ:1overv} again, we find that
\begin{equation}\label{equ:hand2}
\sum_{\substack{\;\;\;v \,\mid\, d^\infty \\[1pt] \;\;\;dv \,>\, x^{2/3}}} \pi(dv, \pm 1; x) \ll \!\!\!\sum_{\substack{\;\;\;v \,\mid\, d^\infty \\[1pt] \;\;\;dv \,>\, x^{2/3}}} \frac{x}{dv} \leq \!\!\!\sum_{\substack{\;\;\;w \,\mid\, d^\infty \\[1pt] \;\;\;w \,>\, x^{2/3}}} \frac{x}{w} \ll x^{1/3} (\omega(d) + 1)(\log x)^{\omega(d)} .
\end{equation}
Putting together~\eqref{equ:hand1}, \eqref{equ:hand2}, and~\eqref{equ:sumadmuboundpi}, taking into account that $\omega(d) \leq \log y$ and $\varphi(d) y \leq \log x$, we obtain~\eqref{equ:firstbound}.
Finally, from~\eqref{equ:sumadmubounddelta}, \eqref{equ:defdeltaUnd}, and~\eqref{equ:1overv}, we get
\begin{equation*}
\sum_{\substack{\;\;\;v \,\mid\, d^\infty \\[1pt] v \,>\, y}} \sum_{a \,\mid\, d} \mu(a) \delta_{U, dv, av} \leq \sum_{\substack{\;\;\;v \,\mid\, d^\infty \\[1pt] v \,>\, y}} \delta_{U, dv, v}
\ll_U \frac1{\varphi(d)}\sum_{\substack{\;\;\;v \,\mid\, d^\infty \\[1pt] v \,>\, y}} \frac1{v^2} \ll \frac{\omega(d) + 1}{\varphi(d)} \cdot \frac{(\log y)^{\omega(d)}}{y} ,
\end{equation*}
as desired.
\end{proof}

\begin{lem}\label{lem:series}
Let $d$ be an odd positive integer with $3 \nmid d$ whenever $\Delta_0 = -3$.
Then
\begin{equation*}
\sum_{\;\;\;v \,\mid\, d^\infty} \sum_{a \,\mid\, d} \mu(a) \delta_{U, dv, av} = \delta_U(d) .
\end{equation*}
\end{lem}
\begin{proof}
For every integer $e$ dividing $d^\infty$, define
\begin{equation*}
S_{d,e,h} := \sum_{\substack{\;\;\;v \,\mid\, d^\infty \\[1pt] e \,\mid\, v}} \sum_{a \,\mid\, d} \frac{\mu(a)(av, h)}{\varphi(dv)av} .
\end{equation*}
The value of $S_{d,1,h}$ was computed in~\cite[Lemma~4]{MR2274151} and a slight modification of the proof (precisely, replacing $(h, d^\infty)$ with $[e, (h, d^\infty)]$ in the last equation) yields
\begin{equation*}
S_{d,e,h} = \frac{(d^\infty, h)}{d\big[(d^\infty, h), e\big]^2} \prod_{p \,\mid\, d}\left(1 - \frac1{p^2}\right)^{-1} .
\end{equation*}
At this point, by \eqref{equ:defdeltaUnd} and considering that $\Delta_0 \mid dv$ if and only if $e \mid v$, where $e := \Delta_0 / (d, \Delta_0)$, we have
\begin{equation*}
\sum_{\;\;\;v \,\mid\, d^\infty} \sum_{a \,\mid\, d} \mu(a) \delta_{U, dv, av} = 
\begin{cases}S_{d, 1, h}  & \text{ if $\Delta > 0$ or $\Delta_0 \not\equiv 1 \!\!\!\pmod 4$ or $\Delta_0 \nmid d^\infty$}\\
S_{d, 1, h} + S_{d, e, h} & \text{ otherwise}
\end{cases}
= \delta_U(d) ,
\end{equation*}
as claimed.
\end{proof}

\begin{proof}[Proof of Theorem~\ref{thm:main}]
Let $A, B > 0$ be the constants of Lemma~\ref{lem:piUnd}.
Assume that $x \geq \exp\big(Be^{8\omega(d)}d^8\big)$ and put $y := (\log x / B)^{1/8} / d$.
Note that $e^{\omega(d)} \leq y \leq \log x / \varphi(d)$ and $\log y \leq \log \log x$, for every $x \gg_B 1$.
By~Lemma~\ref{lem:RUdx}, Lemma~\ref{lem:piUnd}, and Lemma~\ref{lem:series}, we obtain that
\begin{align*}
\mathcal{R}_U(d; x) &= \sum_{\substack{\;\;\;v \,\mid\, d^\infty \\[1pt] v \,\leq\, y}} \sum_{a \,\mid\, d} \mu(a) \pi_{U, dv, av}(x) + O(E_1) \\
&= \sum_{\substack{\;\;\;v \,\mid\, d^\infty \\[1pt] v \,\leq\, y}} \sum_{a \,\mid\, d} \mu(a) \delta_{U, dv, av} \Li(x) + O(E_1) + O_U(E_2) \\
&= \delta_U(d) \Li(x) + O(E_1) + O_U(E_2) + O(E_3) ,
\end{align*}
where, by Lemma~\ref{lem:bounds}, we have
\begin{equation*}
E_1 := \sum_{\substack{\;\;\;v \,\mid\, d^\infty \\[1pt] v \,>\, y}} \sum_{a \,\mid\, d} \mu(a) \pi_{U, dv, av}(x) \ll \frac{x}{\log x} \cdot \frac{\omega(d) + 1}{\varphi(d)} \cdot \frac{(\log y)^{\omega(d)}}{y} \ll \frac{(\omega(d) + 1) d}{\varphi(d)} \cdot \frac{x\,(\log \log x)^{\omega(d)}}{(\log x)^{9/8}}
\end{equation*}
and
\begin{equation*}
E_3 := \sum_{\substack{\;\;\;v \,\mid\, d^\infty \\[1pt] v \,>\, y}} \sum_{a \,\mid\, d} \mu(a) \delta_{U, dv, av} \Li(x) \ll_U \frac{\omega(d) + 1}{\varphi(d)} \cdot \frac{(\log y)^{\omega(d)}}{y} \cdot \Li(x) \ll \frac{(\omega(d) + 1) d}{\varphi(d)} \cdot \frac{x\,(\log \log x)^{\omega(d)}}{(\log x)^{9/8}} ,
\end{equation*}
while, also using the inequality $\tau(d) / d \leq d / \varphi(d)$, we have
\begin{align*}
E_2 &:= \sum_{\substack{\;\;\;v \,\mid\, d^\infty \\[1pt] v \,\leq\, y}} \sum_{a \,\mid\, d} x \exp\big({-A}(\log x)^{1/2} / (dv) \big) \ll x \exp\big({-A}B^{1/8}(\log x)^{3/8} \big) \tau(d) y \\
&\ll x \exp\big({-A}B^{1/8}(\log x)^{3/8} \big) (\log x)^{1/8} \cdot \frac{\tau(d)}{d} \ll \frac{d}{\varphi(d)} \cdot \frac{x}{(\log x)^{9/8}} .
\end{align*}
The result follows.
\end{proof}

\bibliographystyle{amsplain}

\end{document}